\documentclass[12pt,a4paper]{article}
\usepackage{amsmath,amssymb,amsthm,mathrsfs}

\newcommand{\Z}{\ensuremath{\mathbb{Z}}}

\newtheorem{thm}{Theorem}[section]
\newtheorem{prop}[thm]{Proposition}
\newtheorem{lemma}[thm]{Lemma}
\newtheorem{rem}[thm]{Remark}

\title{Lonesum decomposable matrices}
\date{\empty}
\author{Ken Kamano}

\begin{document}

\maketitle

\begin{abstract}

A lonesum matrix is a $(0,1)$-matrix that
is uniquely determined by 
its row and column sum vectors.
In this paper, we introduce
lonesum decomposable matrices and 
study their properties.
We provide a necessary and sufficient condition
for a matrix $A$ to be lonesum decomposable,
and give 
a generating function for
the number $D_k(m,n)$ of $m\times n$ lonesum decomposable matrices of order $k$.
Moreover, by using this generating function
we prove some congruences 
for $D_k(m,n)$ modulo a prime.

\vspace{10pt}


\end{abstract}

\section{Introduction}

A \textit{$(0,1)$-matrix} (resp.~\textit{vector}) 
is a matrix (resp.~vector) in which each entry is 
zero or one.
A $(0,1)$-matrix $A$  is called 
a \textit{lonesum matrix}
if $A$ is uniquely determined by
its row and column sum vectors.
For example, 
a $(0,1)$-matrix with
a row sum vector ${}^t(3,1)$
and a column sum vector $(1,2,1)$
is uniquely determined as the following:
\[  \begin{pmatrix}
1& 1 & 1  \\
0&1 & 0
\end{pmatrix}. \]
Hence, the matrix
$ \begin{pmatrix}
1& 1 & 1  \\
0&1 & 0
\end{pmatrix}$ is a lonesum matrix.
Because $  \begin{pmatrix}
1& 0 & 1  \\
0&1 & 0
\end{pmatrix}$
and $  \begin{pmatrix}
0& 1 & 1  \\
1&0 & 0
\end{pmatrix}$
have the same row and column sum vectors,
they 
are not lonesum matrices.
We denote by $L(m,n)$ the
number of $m\times n $ lonesum matrices.
For simplicity, we set $L(m,0) = L(0,m)=1$ for
any non-negative integer $m$.
It is known that lonesum matrices are related
to certain combinatorial objects.
For example, the number $L(m,n)$ is equal to
the number of acyclic orientations
of the complete bipartite graph $K_{m,n}$
(\cite[Theorem 2.1]{CG}).

An $m \times n$ $(0,1)$-matrix $A=(a_{ij})$ is called a {\it Ferrers matrix} if 
$A$ satisfies the condition
\begin{equation*}\label{eq:Ferrers}
\begin{cases}
 a_{ij}=0 \Rightarrow a_{kj}=0  &(k\ge i),\\ 
 a_{ij}=0 \Rightarrow a_{il}=0  &(l\ge j).
\end{cases}
\end{equation*}
This condition means that all 1 entries of $A$ are placed 
to the upper left of $A$. For example,
the matrix $\begin{pmatrix}
1&1 &1 \\1&0&0
\end{pmatrix}$ is a Ferrers matrix.
Ryser  \cite{R} investigated matrices 
that have fixed row and column sum vectors.
In our setting, his result can be written as follows:
\begin{prop}\label{prop:Ryser}
Let $A$ be a $(0,1)$-matrix.
Then, the following conditions are equivalent:
\begin{itemize}
\item[(i)] $A$ is a lonesum matrix.
\item[(ii)] $A$ does not contain
$\begin{pmatrix} 1 & 0 \\ 0 &1\end{pmatrix}$ or 
$\begin{pmatrix} 0 & 1 \\ 1 &0\end{pmatrix}$
as a submatrix.
\item[(iii)] $A$ is obtained from a Ferrers matrix 
by permutations of rows and columns.
\end{itemize}
\end{prop}

For an integer $k$, Kaneko \cite{K} introduced
poly-Bernoulli numbers $B_n^{(k)}$ of index $k$ 
as 
\begin{equation}
 \frac{{\rm Li}_k(1-e^{-t})}{1- e^{-t}} = \sum_{n=0}^{\infty} B_n^{(k)}\frac{t^n}{n!}, 
\end{equation}
where ${\rm Li}_k(z)$ denotes the $k$-th polylogarithm, defined by
$ {\rm Li}_k(z) := \sum_{n=1}^{\infty} z^n/n^k$.
Brewbaker \cite{B} proved that
the numbers $L(m,n)$ are equal to the poly-Bernoulli numbers
of negative indices:
\begin{equation}\label{eq:Brewbaker}
 L(m,n) = B^{(-m)}_n \ \ \ (m,n\ge 0).
\end{equation} 
The generating function of 
poly-Bernoulli numbers of negative indices 
has been given by Kaneko \cite{K}, 
hence the numbers of lonesum matrices
have the following generating function:
\begin{align} \label{eq:genfun_L}
\sum_{m=0}^{\infty}
\sum_{n=0}^{\infty}
L(m,n) \dfrac{x^m}{m!} \dfrac{y^n}{n!}
= \dfrac{e^{x+y}}{e^x+e^y-e^{x+y}}.
\end{align}
The present author, Ohno, and Yamamoto \cite{KOY}
introduced ``weighted'' lonesum matrices
and a simple proof of 
\eqref{eq:genfun_L} was given (see 
\cite[Proof of Theorem 1]{KOY}).

For $m\times n$ matrices $A$ and $B$, 
we write $A\sim B$ if $A$ is obtained from $B$
by row or column exchanges.
We call a $(0,1)$-matrix $A$ is \textit{lonesum decomposable} if $A$ satisfies the condition
\[ A \sim \begin{pmatrix}
 L_1 &    &  & O \\
 &  L_2   &  &  \\
 &    & \ddots  &  \\
  O   &    & 	&L_{k}
\end{pmatrix},\]
where $L_i$ ($1\le i \le k$) are lonesum matrices.
A lonesum matrix is clearly
lonesum decomposable.
Since a lonesum matrix can be obtained from 
a Ferrers matrix,
a lonesum decomposable matrix $A$ can be
transformed as
\begin{align}\label{eq:lonesumdecom}
  A \sim  \begin{pmatrix}
F_1 &    &  &O &\\
& F_2   &  &  &\\
&    &\ddots  &  & \\	
O&    & &	F_{k} & \\
&    & &	 & O 
\end{pmatrix},
\end{align}
where $F_i$ ($1\le i \le k$) are Ferrers matrices with no zero rows or zero columns.
We call the right-hand side of \eqref{eq:lonesumdecom} the
\textit{decomposition matrix} of $A$ and $k$ the \textit{decomposition order} of $A$.

\begin{prop}
Let $A$ be a lonesum decomposable matrix.
Then the decomposition matrix of $A$ is
uniquely determined up to the order of $F_i$ 
($1\le i \le k$).
In particular, the decomposition order of $A$ 
is uniquely determined.
\end{prop}

\begin{proof}
For a lonesum decomposable matrix $A=(a_{ij})$, 
it follows from Proposition \ref{prop:Ryser} that 
two elements $a_{ij}=1$  and $a_{i' j'}=1$ belong to the same Ferrers block
if and only if $a_{ij}$  and $a_{i' j'}$ do not form a submatrix 
$\begin{pmatrix}
1&0\\ 0 &1
\end{pmatrix}$ or 
$\begin{pmatrix}
0&1\\ 1 &0
\end{pmatrix}$.
Because the type of Ferrers matrix is uniquely determined,
a decomposition matrix of $A$ is also
uniquely determined up to the order of its Ferrers blocks.
\end{proof}

The outline of this paper is as follows.
In Section
\ref{sec:Lonesum decomposable matrices},
we show that a $(0,1)$-matrix $A$ is lonesum decomposable
if and only if $A$ does not contain certain matrices 
as submatrices.
In Section \ref{sec:Generating function of lonesum decomposable matrices},
we give a generating function for the number of 
lonesum decomposable matrices.
In Section \ref{sec:Congruences for $D_k(m,n)$},
we derive some congruences for the numbers of 
lonesum decomposable matrices of order $k$
by using the generating function given
in Section \ref{sec:Generating function of lonesum decomposable matrices}.

\section{Lonesum decomposable matrices}
\label{sec:Lonesum decomposable matrices}
Let us define a $2\times 3$ matrix $U$ as
\[ U:=
\begin{pmatrix}
1 & 1 & 0 \\
1 & 0 & 1
\end{pmatrix}.\]
It can be easily checked that $U$ is not lonesum decomposable.
Let $\mathcal{N}$ be a set of all matrices 
obtained from  $U$ or ${}^tU$
by permutations
of rows and columns.
Namely, the elements of $\mathcal{N}$ are the following
twelve matrices:
\begin{center}
$\begin{pmatrix}
1 & 1 & 0 \\
1 & 0 & 1
\end{pmatrix}$,\ \ 
$\begin{pmatrix}
1 & 0 & 1 \\
1 & 1 & 0
\end{pmatrix}$,\ \ 
$\begin{pmatrix}
1 & 1 & 0 \\
0 & 1 & 1
\end{pmatrix}$,\ \ 
$\begin{pmatrix}
0 & 1 & 1 \\
1 & 1 & 0
\end{pmatrix}$,\ \ 
$\begin{pmatrix}
1 & 0 & 1 \\
0 & 1 & 1
\end{pmatrix}$,\ \ 
$\begin{pmatrix}
0 & 1 & 1 \\
1 & 0 & 1
\end{pmatrix}$,\ \ 

$\begin{pmatrix}
1 & 1  \\
1 & 0 \\
0 & 1
\end{pmatrix}$,\ \ 
$\begin{pmatrix}
1 & 1  \\
0 & 1 \\
1 & 0
\end{pmatrix}$,\ \ 
$\begin{pmatrix}
1 & 0  \\
1 & 1 \\
0 & 1
\end{pmatrix}$,\ \ 
$\begin{pmatrix}
0 & 1  \\
1 & 1 \\
1 & 0
\end{pmatrix}$,\ \ 
$\begin{pmatrix}
1 & 0  \\
0 & 1 \\
1 & 1
\end{pmatrix}$,\ \ 
$\begin{pmatrix}
0 & 1  \\
1 & 0 \\
1 & 1
\end{pmatrix}$.
\end{center}

The following is the first main result of this paper.
\begin{thm}\label{thm:avoidthm}
Let $A$ be a $(0,1)$-matrix.
Then, the following two conditions are equivalent.
\begin{enumerate}
\item $A$ is lonesum decomposable.
\item $A$ does not contain an element of $\mathcal{N}$
as a submatrix.
\end{enumerate}
\end{thm}

\begin{proof}
It is clear that (i) $\Rightarrow$ (ii), and
we show  (ii) $\Rightarrow$ (i).
This statement clearly holds for 
$0\le m,n \le 2$, where $m$ and $n$ are
the numbers of rows and columns of $A$,  respectively.
A transpose of a lonesum decomposable 
matrix is also lonesum decomposable,
hence we only have to prove
that if the statement holds for 
all $m\times n $ matrices,
then it holds for any $m\times (n+1)$ matrix
for $m,n \ge 2$.

Let $A$ be an $m\times (n+1)$ $(0,1)$-matrix
not containing an element of $\mathcal{N}$.
The matrix obtained by removing the $(n+1)$-st column from $A$
is $m\times n$ matrix.
Hence, by the induction assumption, 
the matrix $A$ can be transformed as 
\[    A \sim  
\left( 
\begin{array}{cccc|c}
F_1     &  &O & &\boldsymbol{b}_1\\
&    \ddots  &  & & \vdots  \\	
O&   &	F_{k} & & \boldsymbol{b}_k \\
&     &	 & O  & \boldsymbol{c}
 \end{array}
\right),
\]
where $F_i$ ($1\le i \le k$) are 
Ferrers matrices with no zero rows or columns,
and
$\boldsymbol{b}_i$ ($1\le i \le k$) and 
$\boldsymbol{c}$ are 
$(0,1)$-vectors.
If there exist two non-zero vectors 
$\boldsymbol{b}_i$ and $\boldsymbol{b}_j$ ($i\neq j$),
then the submatrix 
$\begin{pmatrix}
F_i & O   & \boldsymbol{b}_i \\
O   & F_j & \boldsymbol{b}_j 
\end{pmatrix}$
contains a matrix 
$\begin{pmatrix}
1 & 0  & 1\\
0 & 1&  1 
\end{pmatrix}$, and 
this contradicts the assumption that $A$ does not contain
any element of $\mathcal{N}$.
Therefore, there is at most one non-zero vector
in $\boldsymbol{b}_i$ ($1\le i \le k$), 
and we can set 
$\boldsymbol{b}_1 = \cdots = \boldsymbol{b}_{k-1}=\boldsymbol{0}$
without loss of generality.

We consider the two cases where
(i) $\boldsymbol{c}$ has $1$'s  and 
(ii) $\boldsymbol{c}$ has no $1$'s.
\vspace{5pt}

\noindent (i). The case that $\boldsymbol{c}$ has $1$'s.

If the vector $\boldsymbol{b}_k$
has both $0$'s and $1$'s, then 
$\begin{pmatrix}
F_k    & \boldsymbol{b}_k \\
O    & \boldsymbol{c} 
\end{pmatrix}$
contains a matrix 
$\begin{pmatrix}
1 & 1  \\
1 & 0 \\ 
0 & 1
\end{pmatrix}$ or 
$\begin{pmatrix}
1 & 0  \\
1 & 1 \\ 
0 & 1
\end{pmatrix}$,
and this contradicts the assumption that $A$ does not contain
an element of $\mathcal{N}$.
Therefore, 
$\boldsymbol{b}_k= \boldsymbol{1}$ or
$ \boldsymbol{0}$.
If $\boldsymbol{b}_k= \boldsymbol{1}$,
then 
\begin{equation}\label{eq:mainthm1_1}
\left( 
\begin{array}{cc|c}
	F_{k}&      & \boldsymbol{b}_k \\
             & O  & \boldsymbol{c}
 \end{array}
\right) \sim
\left( 
\begin{array}{ccc}
\boldsymbol{1} & F_{k} &     \\
\boldsymbol{c}     &          &O  
 \end{array}
\right) .
\end{equation}
Because the right-hand side of 
\eqref{eq:mainthm1_1}
is a lonesum matrix, the statement holds.
If $\boldsymbol{b}_k= \boldsymbol{0}$,
then 
\begin{equation}\label{eq:mainthm1_2}
\left( 
\begin{array}{cc|c}
	F_{k}&      & \boldsymbol{b}_k \\
             & O  & \boldsymbol{c}
 \end{array}
\right) \sim
\left( 
\begin{array}{ccc}
F_{k} & \boldsymbol{0}  &     \\
          &   \boldsymbol{c}    & O  
 \end{array}
\right) .
\end{equation}
The right-hand side of \eqref{eq:mainthm1_2}
is lonesum decomposable 
of order $2$, and hence the statement again holds.
\vspace{5pt}

\noindent (ii). The case that
$\boldsymbol{c}$ has no $1$'s.

We have
\begin{equation}\label{eq:mainthm1_3}
\left( 
\begin{array}{cc|c}
	F_{k}&      & \boldsymbol{b}_k \\
             & O  & \boldsymbol{c}
 \end{array}
\right) \sim
\left( 
\begin{array}{ccc}
F_{k} & \boldsymbol{b}_k  &     \\
          &   \boldsymbol{0}    & O  
 \end{array}
\right) .
\end{equation}
By Proposition \ref{prop:Ryser},
if the matrix
$(F_k\, \boldsymbol{b}_k)$ is not a lonesum matrix
then it contains $\begin{pmatrix}
1 & 0 \\ 0 & 1 
\end{pmatrix}$ or 
$\begin{pmatrix}
0 & 1 \\ 1 & 0 
\end{pmatrix}$ as a submatrix.
Because $F_k$ has no zero columns, 
the matrix $(F_k\, \boldsymbol{b}_k)$ also contains 
$\begin{pmatrix}
1 & 1 & 0 \\
1 & 0 & 1
\end{pmatrix}$
or
$\begin{pmatrix}
1 & 0 & 1 \\
1 & 1 & 0
\end{pmatrix}$
as a submatrix, 
and this contradicts the assumption that $A$ does not contain
an element of $\mathcal{N}$.
Therefore, the matrix
$(F_k\, \boldsymbol{b}_k)$ is a lonesum matrix
and the statement also holds in this case.
\end{proof}

For a $(0,1)$-matrix $A$, 
we define $\overline{A}$ as
the matrix in which the $0$ and $1$ 
entries of $A$ are inverted.
If $A$ is a lonesum matrix, then 
$\overline{A}$ is also a lonesum matrix.
However, lonesum decomposable matrices 
do not have this property.
For example, 
the matrix
$ V=
\begin{pmatrix}
1 & 0 & 0 \\
0 & 1 & 0
\end{pmatrix}$ is lonesum decomposable, 
but $ \overline{V}=
\begin{pmatrix}
0 & 1 & 1 \\
1 & 0 & 1
\end{pmatrix}\in \mathcal{N}$ is not lonesum decomposable.
The following theorem determines 
when a matrix $A$ satisfies that both 
$A$ and $\overline{A}$ are lonesum decomposable.

\begin{thm}
Let $A$ be a $(0,1)$-matrix.
Then, the following conditions are equivalent.

\begin{enumerate}
	\item Both $A$ and $\overline{A}$ are lonesum decomposable.
	\item $A$ is a lonesum matrix or $A \sim 
	\begin{pmatrix}
	{ \boldsymbol{1}} & { \boldsymbol{O}}  \\
	{ \boldsymbol{O}} & { \boldsymbol{1}}
	\end{pmatrix}
	$, where
	$\boldsymbol{1}$ (resp.~$\boldsymbol{O}$)
	is a matrix whose entries are all $1$ (resp.~$0$).
\end{enumerate}
\end{thm}

\begin{proof}
It is clear that (ii) $\Rightarrow $ (i), and
we only have to prove that (i) $\Rightarrow $ (ii).
Assume that $A$ and $\overline{A}$ are both lonesum decomposable, and let
$k$ be the decomposition order of $A$.
When $k=0$ or $1$,  $A$ is a lonesum matrix.
When $k=2$, 
the matrix $A$ satisfies that 
\[ A\sim  \begin{pmatrix}
L_1 & O \\
O & L_2
\end{pmatrix}, \]
where $L_1$ and $L_2$ are non-zero lonesum matrices.
If $L_1$ or $L_2$ has $0$'s, then
the matrix $\overline{A}$ contains an element of $\mathcal{N}$ as a submatrix, 
and  $\overline{A}$ is not lonesum decomposable.
Therefore, $L_1=\boldsymbol{1}$ and $L_2=\boldsymbol{1}$.
When $k\ge 3$, 
the matrix $A$ contains a $3\times 3$ submatrix  $W$ satisfying
$W\sim  \begin{pmatrix}
1 & 0 & 0 \\
0 & 1 & 0 \\
0 & 0 & 1
\end{pmatrix}$.
This matrix contains  
$ \begin{pmatrix}
1 & 0 & 0 \\
0 & 1 & 0 \\
\end{pmatrix}$,
and this contradicts the condition that $\overline{A}$ is lonesum decomposable.
As a consequence, 
either $A$ is a lonesum matrix or $A \sim 
	\begin{pmatrix}
	{ \boldsymbol{1}} & { \boldsymbol{O}}  \\
	{ \boldsymbol{O}} & { \boldsymbol{1}}
	\end{pmatrix}
	$.
\end{proof}

\section{Generating function of lonesum decomposable matrices}
\label{sec:Generating function of lonesum decomposable matrices}

For a positive integer $k$, let 
$D_k(m,n)$ denote the number of $m\times n $ lonesum decomposable
matrices of decomposition order $k$.
For simplicity, we set
$D_k(m,0) = D_k(0,m) =0$ for $k\ge 1$ and $m\ge 0$,
and $D_0(m,n) = 1$ for $(m,n) \in \Z_{\ge 0}^2$.
Moreover, we define
$D(m,n):= \sum_{k=0}^{\infty} D_k(m,n)$
for $(m,n) \in \Z_{\ge 0}^2$.
This means that $D(m,n)$ is 
the number of all $m\times n$ lonesum decomposable matrices.
We can see that
$D_k(m,n)=0$ for $k>\min(m,n)$ and
$L(m,n)= D_0(m,n)+D_1(m,n)$.
We present tables showing 
$D_1(m,n)$, $D_2(m,n)$, and $D(m,n)$
at the end of this paper.

The generating functions for
$D_k$ and $D$ are given as follows:
\begin{thm}\label{thm:genfun_D_k}
The following equations hold:
\begin{align}\label{thm:genfun_1}
\sum_{m=0}^{\infty}  \sum_{n=0}^{\infty} D_k(m,n)
\frac{x^m}{m!}
\frac{y^n}{n!}
= \frac{e^{x+y}}{k!} 
\left(  \frac{1}{e^x+e^y-e^{x+y}}-1 \right)^k
\ \ \  (k\ge 0).
\end{align}

\begin{align}\label{thm:genfun_2}
\sum_{m=0}^{\infty}  \sum_{n=0}^{\infty} 
D(m,n)
\frac{x^m}{m!}
\frac{y^n}{n!}
= 
\exp\left(  \frac{1}{e^x+e^y-e^{x+y}} + x +y -1 \right).
\end{align}

\end{thm}

\begin{proof}

Let $\tilde{L}(m,n)$ be
the number of  $m\times n$ lonesum matrices with
no zero rows or columns. Here, we set
$\tilde{L}(0,0) =1$
and $\tilde{L}(m,0) = \tilde{L}(0,m) =0 $ for $m>0$. 
Benyi and Hajnal \cite[Theorem 3]{BH1}
mentioned that the generating function of 
$\tilde{L}(m,n)$ is given by
\begin{align}\label{eq:genfun_tildeL}
 \sum_{m=0}^{\infty}  \sum_{n=0}^{\infty} \tilde{L}(m,n) \dfrac{x^m}{m!}\dfrac{y^n}{n!}
= \dfrac{1}{e^x+e^y -e^{x+y}}. 
\end{align}
By definition, it holds that
\begin{align}
L(m,n) =
\sum_{i=0}^m\sum_{j=0}^n
\binom{m}{i}\binom{n}{j}
\tilde{L}(i,j),
\end{align}
and we can also obtain the generating function \eqref{eq:genfun_L}
of $L(m,n)$ from \eqref{eq:genfun_tildeL}.
We note that multiplying 
the generating function
\eqref{eq:genfun_tildeL}  by $e^{x+y}$
means that it allows 
the lonesum matrices to have zero columns or zero rows.

Let $\tilde{D}_k(m,n)$ be the number of $m\times n$ lonesum decomposable matrices of order $k$ with no zero rows and columns.
We set $\tilde{D}_0(m,n)=0$ if 
$(m,n)\neq (0,0)$
and $=1$ if $(m,n)=(0,0)$.
When $k=1$, we have
$\tilde{D}_1(m,n) 
= \tilde{L}(m,n)$ if $(m,n)\neq (0,0)$
and $=0$ if $(m,n)=(0,0)$.
Therefore, we have
\[ 
\sum_{m=0}^{\infty}  \sum_{n=0}^{\infty} 
\tilde{D}_1(m,n) \dfrac{x^m}{m!} \dfrac{y^n}{n!}
= \frac{1}{e^x+e^y-e^{x+y}}-1.\]
In general, 
the generating function of $\tilde{D}_k$ can be given by 
\begin{align}\label{eq:tilde{D_k}genfun}
\sum_{m=0}^{\infty}  \sum_{n=0}^{\infty} 
\tilde{D}_k(m,n) \dfrac{x^m}{m!} \dfrac{y^n}{n!}
= \frac{1}{k!} 
\left(  \frac{1}{e^x+e^y-e^{x+y}}-1 \right)^k\ \ \ 
(k\ge 0).
\end{align}
The generating function of $D_k$
can be obtained by
multiplying \eqref{eq:tilde{D_k}genfun} by $e^{x+y}$,
hence we obtain \eqref{thm:genfun_1}.
Equation \eqref{thm:genfun_2} follows
immediately from \eqref{thm:genfun_1}.
\end{proof}

\begin{rem}
Ju and Seo \cite{JS}
studied generating functions for the
number of matrices not including various $2\times 2$ matrices.
Theorem \ref{thm:genfun_D_k} gives
a similar result on
matrices that do not include the elements of
$\mathcal{N}$.
\end{rem}

It is known that the numbers $L(m,n)$ (or the poly-Bernoulli numbers of negative indices) 
satisfy a recurrence relation (e.g.~\cite[Prop.~14.3 and 14.4]{AIK}).
Our numbers $D_k(m,n)$ also satisfy a recurrence relation.
\begin{prop}
For $k\ge 1$ and $m,n\ge 0$, we have
\begin{align*}
&D_k(m+1,n) \\
&= D_k(m,n)
+\sum_{l=0}^{n-1} \binom{n}{l}
\left( (k-1)D_k(m,l) + D_{k-1} (m,l)  + D_k(m,l+1)  \right).
\end{align*}
\end{prop}

\begin{proof}
Let $G_k (x,y):= \dfrac{e^{x+y}}{k!} \left(  \dfrac{1}{e^x+e^y -e^{x+y}} -1  \right)^k $.
By direct calculations, we can verify that
\begin{align}\label{eq:recurrence_gf}
 \dfrac{\partial }{\partial  x } G_k =
G_k + (e^y-1) \left(  (k-1)G_k + G_{k-1} + \dfrac{\partial }{\partial  y} G_k  \right).
\end{align}
By comparing the coefficients
of both sides of \eqref{eq:recurrence_gf}, 
we obtain the proposition.
\end{proof}

To conclude this section, we give
a relation between $D_k(m,n)$ and the
poly-Bernoulli polynomials.
For any integers $k_1,\ldots ,k_r$, 
we define the multi-poly Bernoulli(-star) polynomials
$B_{n,\star}^{(k_1,\ldots, k_r)}(x)$ by
\begin{equation}\label{eq:mpBsp}
e^{-xt}
 \dfrac{\text{Li}^{\star}_{k_1,\ldots, k_r}(1-e^{-t}) }{1-e^{-t}}
= \sum_{n=0}^{\infty} 
B_{n,\star}^{(k_1,\ldots, k_r)}(x)  \dfrac{t^n}{n!},
\end{equation}
where 
\[ \text{Li}^{\star}_{k_1,\ldots, k_r}(z)
:= \sum_{1\le m_1\le \cdots \le m_r}
\dfrac{z^{m_r}}{m_1^{k_1} \cdots m_r^{k_r}}.
\]
These polynomials have been introduced by Imatomi \cite[\S6]{I},
but they were defined there 
with $e^{-xt}$ replaced by $e^{xt}$
in \eqref{eq:mpBsp}.
When $r=1$, 
the polynomial $B_{n,\star}^{(k)}(x)$ coincides with the  
$n$-th poly-Bernoulli polynomial $B_n^{(k)}(x)$ defined by 
\begin{align*}
e^{-xt} \dfrac{\text{Li}_k(1-e^{-t})}{1-e^{-t}}
= \sum_{n=0}^{\infty} B_n^{(k)}(x) \dfrac{t^n}{n!}
\end{align*}
(see e.g., Coppo-Candelpergher \cite{CC}).

\begin{prop}
For integers $k,m,n\ge 0$, we have
\[ D_k(m,n) =
\dfrac{(-1)^k}{k!} \left(  1+\sum_{i=1}^k  \binom{k}{i} (-1)^i B_{n,\star}^{ ( \scriptsize\overbrace{ 0,\ldots , 0}^{i-1}, -m) } (i-1)   \right).
\]
\end{prop}

\begin{proof}

For an integer $i \ge 1$, we have
\begin{align*}
 \frac{e^{x+y}}{(e^x+e^y-e^{x+y})^i} 
&= e^{x+y} \left( \dfrac{1}{e^y(1-e^x(1-e^{-y}))} \right)^i\\
&= e^{x+y} e^{-iy} 
     \sum_{l_1,\ldots , l_i \ge 0}
       e^{(l_1+\cdots +l_i)x}  (1-e^{-y})^{l_1+\cdots +l_i}     \\
&= e^{-(i-1)y} 
     \sum_{l_1,\ldots , l_i \ge 0}
       e^{(l_1+\cdots +l_i+1)x}  (1-e^{-y})^{l_1+\cdots +l_i+1}  \dfrac{1}{1-e^{-y}}    \\
&= e^{-(i-1)y}  \sum_{m=0}^{\infty} 
     \sum_{l_1,\ldots , l_i \ge 0}
       \dfrac{(1-e^{-y})^{l_1+\cdots +l_i+1}}
                {(l_1+\cdots +l_i+1)^{-m}}  \dfrac{1}{1-e^{-y}}  \dfrac{x^m}{m!}   \\
&= e^{-(i-1)y}  \sum_{m=0}^{\infty} 
      \dfrac{\text{Li}^{\star}_{0,\ldots , 0,-m}(1-e^{-y})  }{1-e^{-y}}  \dfrac{x^m}{m!}   \\
&= \sum_{m=0}^{\infty}  \sum_{n=0}^{\infty}
      B_{n,\star}^{ ( \scriptsize\overbrace{ 0,\ldots , 0}^{i-1}, -m) } (i-1)   \dfrac{x^m}{m!} \dfrac{x^n}{n!}.
 \end{align*} 
From this formula and the binomial expansion,
we obtain that
\begin{align*}
& \dfrac{e^{x+y}}{k!} \left( \dfrac{1}{e^x+e^y-e^{x+y}}-1\right)^k\\
&= \dfrac{(-1)^k}{k!} 
\left( e^{x+y}+ \sum_{i=1}^{k} \binom{k}{i}(-1)^i
\dfrac{ e^{x+y} }{(e^x+e^y-e^{x+y})^i}    \right) \\
&=\sum_{m=0}^{\infty}  \sum_{n=0}^{\infty} \dfrac{(-1)^k}{k!} 
\left( 1+ \sum_{i=1}^{k}\binom{k}{i}(-1)^i 
B_{n,\star}^{ ( \scriptsize\overbrace{ 0,\ldots , 0}^{i-1}, -m) } (i-1) 
   \right)  \dfrac{x^m}{m!} \dfrac{x^n}{n!},
\end{align*}
and this proves the proposition.
\end{proof}

\begin{rem}

Kaneko, Sakurai, and Tsumura \cite{KST} introduced
a sequence $\mathscr{B}_m^{(-l)} (n)$ as
 \[ \mathscr{B}_m^{(-l)} (n) :=
\sum_{j=0}^n \left[  n \atop j \right]  
B_m^{(-l-j)}(n)\ \ \ (l,m,n\in \Z_{\ge 0}), \]
where $\left[ n \atop j \right]$ are
the Stirling numbers of the first kind.
They proved that this sequence 
has the following simple generating function:
\begin{align}\label{eq:KST_generating}
\sum_{l=0}^{\infty}
\sum_{m=0}^{\infty}
 \mathscr{B}_m^{(-l)} (n) \frac{x^l}{l!} \frac{y^m}{m!}
 = \frac{n! e^{x+y}}{(e^x + e^y - e^{x+y})^{n+1}}.
\end{align}
By using this formula, we can also give 
an expression for $D_k(m,n)$ in 
terms of poly-Bernoulli polynomials:
\begin{align}\label{eq:D_k_polybernoulli}
D_k (m,n)=
\frac{(-1)^k}{k!} 
\left(  
1+ \sum_{i=0}^{k-1} \frac{(-1)^{i+1} }{i !}
\binom{k}{i+1}  \sum_{j=0}^i
\left[ i \atop j \right] B_n^{(-m-j)} (i)
\right).
\end{align}
\end{rem}

%
%
%

\section{Congruences for $D_k(m,n)$}
\label{sec:Congruences for $D_k(m,n)$}

It is known that the numbers of $m\times n$ lonesum matrices (or poly-Bernoulli numbers of negative indices) have 
the following expression:
\begin{align*}
L(m,n) = 
\sum_{j=0}^{\min(m,n)}
(j!)^2 \left\{  m+1 \atop j+1  \right\}
\left\{  n+1 \atop j+1  \right\},
\end{align*}
where $\left\{ m \atop j \right\}$ 
are the Stirling numbers of the second kind
(see e.g., \cite{AK} \cite{B}).
We note that $\left\{  m \atop j  \right\}=0$ for $j>m\ge 1$.
The following proposition says that
the numbers $D_k(m,n)$ also have a similar expression.

\begin{prop}\label{prop:D_k_stirling}
For integers $k \ge 1$ and $m,n \ge 0$ we have
\begin{align}\label{eq:D_k_stirling}
D_k (m,n)=
\frac{1}{k!} \sum_{j=k}^{\min(m, n)}
\binom{j-1}{k-1}
\left(j! \right)^2
\left\{m+1 \atop j+1 \right\}
\left\{n+1 \atop j+1 \right\}.
\end{align}
\end{prop}

\begin{proof}
The generating function for $D_k(m,n)$ can be transformed as 
\begin{align*}
&\frac{e^{x+y}}{k!}
\left( \dfrac{1}{e^x+e^y-e^{x+y}}-1   \right)^k\\
&=
\frac{e^{x+y}}{k!}
\left( \dfrac{(e^x-1)(e^y-1)}{ 1-(e^x-1)(e^y-1) }  \right)^k \\
&=
\frac{e^{x+y}}{k!}
\sum_{j=k}^{\infty} \binom{j-1}{k-1}
(e^x-1)^j(e^y-1)^j\\
&=
\frac{1}{k!}
\sum_{j=k}^{\infty} \binom{j-1}{k-1}
\dfrac{1}{(j+1)^2}
\frac{d}{dx}(e^x-1)^{j+1}  \frac{d}{dy}(e^y-1)^{j+1}.
\end{align*}
Because
\[ (e^z-1)^m
= m! \sum_{n=m}^{\infty} 
  \left\{n \atop m \right\} \dfrac{z^n}{n!},
\]
we have
\begin{align*}
&\frac{e^{x+y}}{k!}
\left( \dfrac{1}{e^x+e^y-e^{x+y}}-1   \right)^k\\
&=
\frac{1}{k!}
\sum_{j=k}^{\infty} \binom{j-1}{k-1}
(j!)^2
\sum_{l=j}^{\infty}\sum_{m=j}^{\infty}
  \left\{l+1 \atop j+1 \right\}
  \left\{m+1 \atop j+1 \right\}
\dfrac{x^l}{l!} \dfrac{y^m}{m!}.
\end{align*}
Therefore, we obtain \eqref{eq:D_k_stirling}.
\end{proof}

By using this expression, 
we give some congruences for 
$D_k(m,n)$ modulo a prime.
We first recall the following lemma
in order to prove them.
All of the formulas are deduced from the 
well-known identities
\[
\left\{m \atop k \right\}
=\left\{m-1 \atop k-1 \right\}
+k\left\{m-1 \atop k \right\},\ \ \ \ 
\left\{m \atop k \right\}
=\dfrac{1}{k!} \sum_{n=1}^k (-1)^{k-n} \binom{k}{n}n^m,
\]
and we omit their proofs.

\begin{lemma}\label{lemma:stirlinglemma}
Let $p$ be a prime.
\begin{enumerate}
\item
For positive integers $m$ and $m'$ with $m\equiv m' \pmod{p-1}$
and $0\le i \le p$,
we have 
$\displaystyle \left\{m \atop i \right\} \equiv
  \left\{m' \atop i \right\}  \pmod{p}$.

\item $\displaystyle \left\{p \atop i \right\} \equiv
0 \pmod{p}$ for $2\le i \le p-1$.

\item $\displaystyle \left\{m \atop 2 \right\} 
=2^{m-1}-1$ for $m\ge 1$.

\end{enumerate}

\end{lemma}

\begin{thm}
Let $k$, $m$, $m'$, $n$, and $n'$ be positive integers.
For any prime $p$, the following congruences hold:

\begin{enumerate}
\item If $ k \ge p $, then
\begin{equation}\label{eq:congruence1}
 D_k(m,n) \equiv 0 \pmod{p}. 
\end{equation}

\item If $m\equiv m' $ and $n\equiv n' \pmod{p-1} $,
then
\begin{equation}\label{eq:congruence2}
D_k(m,n) \equiv D_k(m',n') \pmod{p}.
\end{equation}

\item 
If $p>k$, then 
\begin{equation}\label{eq:congruence3}
D_k(p-1,n) \equiv
\begin{cases}
0 & ( n\not \equiv 0\pmod{p-1} ) \\
\displaystyle \frac{(-1)^{k-1}}{(k-1)!} &( n \equiv 0\pmod{p-1} ) 
\end{cases}
 \pmod{p}.
\end{equation}

\item 
\begin{equation}\label{eq:congruence4}
D_k(p,n) \equiv 
\begin{cases}
2^n-1 & (k=1)\\
0 & (k\ge 2)
\end{cases}
\pmod{p}.
\end{equation}

\end{enumerate}
\end{thm}

\begin{proof}
\begin{enumerate}
\item 
If $k\ge p$, 
then
$( j!)^2/k! \equiv 0 \pmod{p}$
in \eqref{eq:D_k_stirling}, 
and this proves that
$D_k(m,n) \equiv 0\pmod{p}$.

\item 
By (i), when $k\ge p$ both sides of 
\eqref{eq:congruence2} vanish modulo $p$
and the congruence holds.
We may assume that $p>k$.
By the symmetric property $D_k(m,n)
= D_k(n,m)$, 
we only have to show that
$D_k(m+p-1,n) \equiv D_k(m,n) \pmod{p}$.
By Proposition \ref{prop:D_k_stirling}, we have
\begin{align}\label{eq:D_k(m+p-1)}
D_k(m+p-1,n) 
&=
\sum_{j=k}^{\min(m+p-1,n)} \binom{j-1}{k-1}
\frac{(j!)^2}{k!} 
  \left\{m+p \atop j+1 \right\}
  \left\{n+1 \atop j+1 \right\}.
\end{align} 
The terms for $j\ge m+1$ in \eqref{eq:D_k(m+p-1)} vanish modulo $p$.
In fact, 
if $m+1 \le j  \le p-1$ then
$ \left\{m+p \atop j+1 \right\}  \equiv \left\{m+1 \atop j+1 \right\}\equiv 0
\pmod{p}$
by Lemma \ref{lemma:stirlinglemma} (i),
and if $j\ge p$ then $j! \equiv 0 \pmod{p}$.
Consequently, we have
\begin{align*}
D_k(m+p-1,n) 
&\equiv 
\sum_{j=k}^{\min(m,n)} \binom{j-1}{k-1}
\frac{(j!)^2}{k!} 
  \left\{m+1 \atop j+1 \right\}
  \left\{n+1 \atop j+1 \right\} \pmod{p},
\end{align*} 
and this is equal to $D_k(m,n)$.

\item 
By (ii), we only have to consider the cases
with $1\le n\le p-1$.
By Proposition \ref{prop:D_k_stirling}, we have
\begin{align}\label{eq:D_k(p-1,n)}
D_k(p-1,n) 
=
\sum_{j=k}^{\min(p-1,n)} \binom{j-1}{k-1}
\frac{(j!)^2}{k!} 
  \left\{p \atop j+1 \right\}
  \left\{n+1 \atop j+1 \right\}.
\end{align} 

If $n\le p-2$, then
$  \left\{p \atop j+1 \right\} \equiv 0 \pmod{p} $
by Lemma \ref{lemma:stirlinglemma} (ii), 
and $D_k(p-1,n)\equiv 0\pmod{p}$.
If $n=p-1$, then only the term for $j=p-1$ in 
\eqref{eq:D_k(p-1,n)} remains, and 
\[ D_k(p-1,n)\equiv 
\binom{p-2}{k-1}\frac{((p-1)!)^2}{k!}
\equiv 
\frac{(-1)^{k-1}}{(k-1)!}
\pmod{p} .\]
The final equivalence is derived from
the congruence $ \binom{p-2}{k-1} \equiv 
(-1)^{k-1} k$
and Wilson's theorem, which states that
$(p-1)! \equiv -1 \pmod{p}$.

\item
By Proposition \ref{prop:D_k_stirling}, we have
\begin{align}\label{eq:D_k(p,n)}
D_k(p,n) 
&=
\sum_{j=k}^{\min(p,n)} \binom{j-1}{k-1}
\frac{(j!)^2}{k!} 
  \left\{p+1 \atop j+1 \right\}
  \left\{n+1 \atop j+1 \right\}.
\end{align} 
When $j=p$, 
we have $(j!)^2/k! \equiv 0\pmod{p}$.
When $2 \le j \le p-1$,  we have
$  \left\{p+1 \atop j+1 \right\}
  = 
 \left\{p \atop j \right\}
 +(j+1) \left\{p \atop j+1 \right\}
\equiv 0 \pmod{p}$
because of Lemma \ref{lemma:stirlinglemma} (ii).
Therefore, the congruence $D_k(p,n)\equiv 
0 \pmod{p}$ holds 
for $k\ge 2$.

If $k=1$, then the term for $j=1$ in \eqref{eq:D_k(p,n)} remains, and 
$D_k(p,n) \equiv 
 \left\{p+1 \atop 2 \right\}
 \left\{n+1 \atop 2 \right\}
= (2^{p}-1)(2^n-1)
\equiv 2^n-1 \pmod{p}$
by Lemma \ref{lemma:stirlinglemma} (iii) and Fermat's little theorem.
\end{enumerate}
\end{proof}

\begin{table}[h]
  \caption{$D_1(m,n)$}
  \label{D1}
  \centering
  \begin{tabular}{c|cccccc}
$m\backslash n$ & 0 & 1 & 2 & 3 & 4 & 5  \\ \hline
0 & 0 & 0 & 0 & 0 & 0 & 0  \\
1 & 0 & 1 & 3 & 7 & 15 & 31  \\
2 & 0 & 3 & 13 & 45 & 145 & 453  \\
3 & 0 & 7 & 45 & 229 & 1065 & 4717  \\
4 & 0 & 15 & 145 & 1065 & 6901 & 41505  \\
5 & 0 & 31 & 453 & 4717 & 41505 & 329461 \\
   \end{tabular}
\end{table}

\begin{table}[h]
  \caption{$D_2(m,n)$}
  \label{D2}
  \centering
  \begin{tabular}{c|cccccc}
$m\backslash n$ & 0 & 1 & 2 & 3 & 4 & 5  \\ \hline
0 & 0 & 0 & 0 & 0 & 0 & 0  \\
1 & 0 & 0 & 0 & 0 & 0 & 0  \\
2 & 0 & 0 & 2 & 12 & 50 & 180  \\
3 & 0 & 0 & 12 & 108 & 660 & 3420  \\
4 & 0 & 0 & 50 & 660& 5714 & 40860 \\
5 & 0 & 0 & 180 & 3420& 40860 & 391500  \\
   \end{tabular}
\end{table}

\begin{table}[h]
	\caption{$D(m,n)$}
	\label{D}
	\centering
	\begin{tabular}{c|cccccc}
		$m\backslash n$ & 0 & 1 & 2 & 3 & 4 & 5  \\ \hline
		0 & 1 & 1 & 1 & 1 & 1 & 1  \\
		1 & 1 & 2 & 4 & 8 & 16 & 32  \\
		2 & 1 & 4 & 16 & 58 & 196 & 634  \\
		3 & 1 & 8 & 58 & 344 & 1786 & 8528  \\
		4 & 1 & 16 & 196 & 1786& 13528 & 90946  \\
		5 & 1 & 32 & 634 & 8528& 90446 & 833432  \\
	\end{tabular}
\end{table}

\vspace{10pt}

\noindent \textbf{Acknowledgements}

\noindent This work was supported by Grant-in-Aid for Young Scientists (B)
from JSPS KAKENHI (16K17583).

\vspace{10pt}

\noindent \textbf{Address}: Department of Mathematics, Osaka Institute of Technology\\
5-16-1, Omiya, Asahi, Osaka 535-8585, Japan\\

\noindent \textbf{E-mail}: ken.kamano@oit.ac.jp

\end{document}